\documentclass[12 pt]{amsart}
\usepackage{amsmath,amssymb,amsthm,hyperref}
\usepackage{a4wide}
\usepackage{color}
\usepackage{enumerate}
\usepackage{cleveref}

\usepackage{tikz}
\usetikzlibrary{decorations.pathreplacing}
\usetikzlibrary{shapes,snakes}

 \newcounter{case}
 \newenvironment{case}[1][\unskip]{\refstepcounter{case}\normalfont
 \medbreak \noindent  \textbf{Case \thecase\ #1.\ }}{\unskip\upshape}
 \renewcommand{\thecase}{\arabic{case}}

\newtheorem{thm}{Theorem}[section]
\newtheorem*{thm*}{Theorem}
\newtheorem*{conj*}{Conjecture}

\newtheorem{lem}[thm]{Lemma}

\theoremstyle{remark}

\theoremstyle{definition}

\newtheorem{defn}[thm]{Definition}

\newcounter{claim}[thm]
\renewcommand{\theclaim}{\noindent{(\arabic{claim})}}
\newcommand{\claim}{\medskip\refstepcounter{claim}{\bf\theclaim} \emph{}}

\title{On generalised Petersen graphs of girth 7 that have cop number 4}

\author{Harmony Morris}

\address{Winston Churchill High School\\
1605 15 Ave N\\
Lethbridge, AB\\
T1H 1W4\\
Canada}

\email{harmony.morris23@lethsd.ab.ca}

\author{Joy Morris}

\address{Department of Mathematics and Computer Science\\
University of Lethbridge\\
Lethbridge, AB\\
T1K 3M4\\
Canada}
\thanks{The second author was supported by the Natural Science and Engineering Research Council of Canada (grant RGPIN-2017-04905).}

\email{joy.morris@uleth.ca}
\begin{document}

\begin{abstract}
We show that if $n=7k/i$ with $i \in \{1,2,3\}$ then the cop number of the generalised Petersen graph $GP(n,k)$ is $4$, with some small previously-known exceptions. It was previously proved by Ball et al.~(2015) that the cop number of any generalised Petersen graph is at most $4$. The results in this paper explain all of the known generalised Petersen graphs that actually have cop number $4$ but were not previously explained by Morris et al.~in a recent preprint, and places them in the context of infinite families. (More precisely, the preprint by Morris et al.~explains all known generalised Petersen graphs with cop number $4$ and girth $8$, while this paper explains those that have girth $7$.)
\end{abstract}

\maketitle

\section{Introduction}

Cops and robbers is a game that can be played on any graph. There are two players: one playing the cops and the other playing the robber. The two players take turns, with the cop going first. On their first turns, each player chooses a vertex on which to place each of their pieces. On all subsequent turns, they may move any or all of their pieces to any neighbouring vertex. The object of the game for the cops is to capture the robber by landing on the same vertex as them. The object for the robber is to avoid being captured forever.

This game first appears in~\cite{nowa}~and~\cite{quilliot}, and has been studied extensively on many families of graphs (see~\cite{anthony} as an excellent reference). 

We are interested in studying this game only on the well-known family of generalised Petersen graphs. The \emph{generalised Petersen graph} $GP(n,k)$ is a graph on $2n$ vertices whose vertex set is the union of $A=\{a_0, \ldots, a_{n-1}\}$ and $B=\{b_0, \ldots, b_{n-1}\}$. The edges have one of three possible forms:
\begin{itemize}
\item $\{a_i, a_{i+1}\}$;
\item $\{a_i,b_i\}$; and
\item $\{b_i, b_{i+k}\}$,
\end{itemize}
where subscripts are calculated modulo $n$. This family generalises the Petersen graph, which is $GP(5,2)$. We require $n \ge 5$, and $k<n/2$. This ensures that the graphs are cubic, and avoids some isomorphic copies of graphs.

The girth of a generalised Petersen graph is at most $8$, and is well-understood from the parameters $n$ and $k$ (see~\cite[Theorem 5]{boben}). Up to isomorphisms, a generalised Petersen graph has girth $7$ if and only if its parameters satisfy one of the following conditions:
\begin{itemize}
\item $n =7k/i$ where $i \in \{1,2,3\}$;
\item $k=4$;
\item $n=2k+3$; or
\item $n=3k\pm 2$.
\end{itemize} 

In~\cite{ball}, Ball et al.~showed that the cop number of every generalised Petersen graph is at most $4$. They also provided a list of all generalised Petersen graphs with $n \le 40$ that attain this bound. (Minor corrections to this list were included in~\cite{MRS}, which also proved that every generalised Petersen graph of girth $8$ that appears on this list falls into an infinite class of generalised Petersen graphs of girth $8$ that have cop number $4$.) Almost all graphs on this list have girth $8$, but there are three graphs of girth $7$ on the list: 
\begin{itemize}
\item $GP(28,8)$;
\item $GP(35,10)$; and
\item $GP(35,15)$.
\end{itemize}
Notably, all of these graphs have parameters of the form $n=7k/2$ or $n=7k/3$. In this paper, we show that with the exception of the generalised Petersen graph with $n<42$ that do not appear above, every generalised Petersen graph whose parameters have the form $n=7k/i$ where $i \in \{1,2,3\}$ has cop number $4$. Thus, in light of the results of~\cite{MRS}, we show that all known generalised Petersen graphs with cop number $4$ are included in infinite families that have this property.

The neighbourhood (up to distance $4$) of an arbitrary vertex $a_i \in A$ is shown in Figure~\ref{figA}, while that for an arbitrary vertex $b_i \in B$ is shown in Figure~\ref{figB}. Note that at distance $4$, some vertices may be the same as others; we have used paired copies of a shape on the corresponding nodes to indicate the vertices that are in fact a single vertex. In our results, we will assume that either $n \ge 42$, or $(n,k) \in \{(28,8),(35,10),(35,15)\}$. Careful calculations (left to the reader) verify that this assumption avoids any additional vertices at distance $4$ being the same as each other. This will be of critical importance in our proofs.

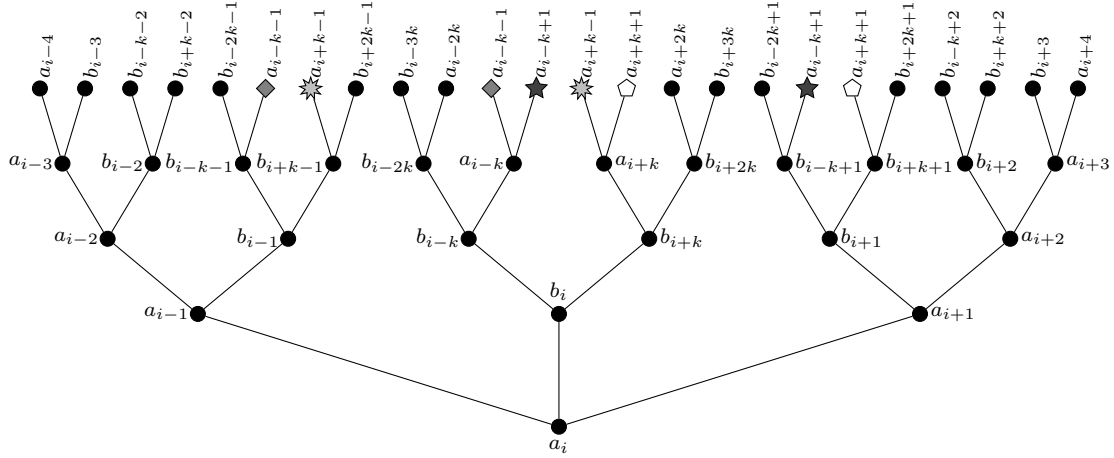
\begin{figure}\caption{The vertices within distance $4$ of $a_i$, in our families of graphs.}
\label{figA}

\begin{tikzpicture}

\node[label={[label distance=-2 pt,font=\tiny]below:$a_i$}] (po1nt) at (0,-.5) [circle,draw,fill=black, inner sep=2pt] {};
\node[label={[label distance=-4 pt,font=\tiny]left:$a_{i-1}$}] (po2nt) at (-4.8,1) [circle,draw,fill=black, inner sep=2pt] {};
\node[label={[label distance=-2 pt,font=\tiny]above:$b_{i}$}] (po3nt) at (0,1) [circle,draw,fill=black, inner sep=2pt] {};
\node[label={[label distance=-3 pt,font=\tiny]right:$a_{i+1}$}] (po4nt) at (4.8,1) [circle,draw,fill=black, inner sep=2pt] {};
\node[label={[label distance=-4 pt,font=\tiny]left:$a_{i-2}$}] (po5nt) at (-6,2) [circle,draw,fill=black, inner sep=2pt] {};
\node[label={[label distance=-4 pt,font=\tiny]left:$b_{i-1}$}] (po6nt) at (-3.6,2) [circle,draw,fill=black, inner sep=2pt] {};
\node[label={[label distance=-4 pt,font=\tiny]left:$b_{i-k}$}] (po7nt) at (-1.2,2) [circle,draw,fill=black, inner sep=2pt] {};
\node[label={[label distance=-3 pt,font=\tiny]right:$b_{i+k}$}] (po8nt) at (1.2,2) [circle,draw,fill=black, inner sep=2pt] {};
\node[label={[label distance=-3 pt,font=\tiny]right:$b_{i+1}$}] (po9nt) at (3.6,2) [circle,draw,fill=black, inner sep=2pt] {};
\node[label={[label distance=-3 pt,font=\tiny]right:$a_{i+2}$}] (po10nt) at (6,2) [circle,draw,fill=black, inner sep=2pt] {};
\node[label={[label distance=-4 pt,font=\tiny]left:$a_{i-3}$}] (po11nt) at (-6.6,3) [circle,draw,fill=black, inner sep=2pt] {};
\node[label={[label distance=-4 pt,font=\tiny]left:$b_{i-2}$}] (po12nt) at (-5.4,3) [circle,draw,fill=black, inner sep=2pt] {};
\node[label={[label distance=-4 pt,font=\tiny]left:$b_{i-k-1}$}] (po13nt) at (-4.2,3) [circle,draw,fill=black, inner sep=2pt] {};
\node[label={[label distance=-4 pt,font=\tiny]left:$b_{i+k-1}$}] (po14nt) at (-3,3) [circle,draw,fill=black, inner sep=2pt] {};
\node[label={[label distance=-4 pt,font=\tiny]left:$b_{i-2k}$}] (po15nt) at (-1.8,3) [circle,draw,fill=black, inner sep=2pt] {};
\node[label={[label distance=-4 pt,font=\tiny]left:$a_{i-k}$}] (po16nt) at (-.6,3) [circle,draw,fill=black, inner sep=2pt] {};
\node[label={[label distance=-3 pt,font=\tiny]right:$a_{i+k}$}] (po17nt) at (0.6,3) [circle,draw,fill=black, inner sep=2pt] {};
\node[label={[label distance=-3 pt,font=\tiny]right:$b_{i+2k}$}] (po18nt) at (1.8,3) [circle,draw,fill=black, inner sep=2pt] {};
\node[label={[label distance=-3 pt,font=\tiny]right:$b_{i-k+1}$}] (po19nt) at (3,3) [circle,draw,fill=black, inner sep=2pt] {};
\node[label={[label distance=-3 pt,font=\tiny]right:$b_{i+k+1}$}] (po20nt) at (4.2,3) [circle,draw,fill=black, inner sep=2pt] {};
\node[label={[label distance=-3 pt,font=\tiny]right:$b_{i+2}$}] (po21nt) at (5.4,3) [circle,draw,fill=black, inner sep=2pt] {};
\node[label={[label distance=-3 pt,font=\tiny]right:$a_{i+3}$}] (po22nt) at (6.6,3) [circle,draw,fill=black, inner sep=2pt] {};
\node[label={[font=\tiny,rotate=90]right:$a_{i-4}$}] (po23nt) at (-6.9,4) [circle,draw,fill=black, inner sep=2pt] {};
\node[label={[font=\tiny,rotate=90]right:$b_{i-3}$}] (po24nt) at (-6.3,4) [circle,draw,fill=black, inner sep=2pt] {};
\node[label={[font=\tiny,rotate=90]right:$b_{i-k-2}$}] (po25nt) at (-5.7,4) [circle,draw,fill=black, inner sep=2pt] {};
\node[label={[font=\tiny,rotate=90]right:$b_{i+k-2}$}] (po26nt) at (-5.1,4) [circle,draw,fill=black, inner sep=2pt] {};
\node[label={[font=\tiny,rotate=90]right:$b_{i-2k-1}$}] (po27nt) at (-4.5,4) [circle,draw,fill=black, inner sep=2pt] {};
\node[label={[font=\tiny,rotate=90]right:$a_{i-k-1}$}] (po28nt) at (-3.9,4) [diamond,draw,fill=black!50, inner sep=1.75pt, minimum width=1.75pt] {};
\node[label={[font=\tiny,rotate=90]right:$a_{i+k-1}$}] (po29nt) at (-3.3,4) [star,star points=9,star point ratio=2,draw,fill=black!25, inner sep=1.5pt, minimum width=1.5pt] {};
\node[label={[font=\tiny,rotate=90]right:$b_{i+2k-1}$}] (po30nt) at (-2.7,4) [circle,draw,fill=black, inner sep=2pt] {};
\node[label={[font=\tiny,rotate=90]right:$b_{i-3k}$}] (po31nt) at (-2.1,4) [circle,draw,fill=black, inner sep=2pt] {};
\node[label={[font=\tiny,rotate=90]right:$a_{i-2k}$}] (po32nt) at (-1.5,4) [circle,draw,fill=black, inner sep=2pt] {};
\node[label={[font=\tiny,rotate=90]right:$a_{i-k-1}$}] (po33nt) at (-.9,4) [diamond,draw,fill=black!50, inner sep=1.75pt, minimum width=1.75pt] {};
\node[label={[font=\tiny,rotate=90]right:$a_{i-k+1}$}] (po34nt) at (-.3,4) [star,star points=5,star point ratio=2,draw,fill=black!75, inner sep=1.5pt, minimum width=1.5pt] {};
\node[label={[font=\tiny,rotate=90]right:$a_{i+k-1}$}] (po35nt) at (.3,4) [star,star points=9,star point ratio=2,draw,fill=black!25, inner sep=1.5pt, minimum width=1.5pt] {};
\node[label={[font=\tiny,rotate=90]right:$a_{i+k+1}$}] (po36nt) at (.9,4) [regular polygon,regular polygon sides=5,draw,fill=white, inner sep=2pt, minimum width=2pt] {};
\node[label={[font=\tiny,rotate=90]right:$a_{i+2k}$}] (po37nt) at (1.5,4) [circle,draw,fill=black, inner sep=2pt] {};
\node[label={[font=\tiny,rotate=90]right:$b_{i+3k}$}] (po38nt) at (2.1,4) [circle,draw,fill=black, inner sep=2pt] {};
\node[label={[font=\tiny,rotate=90]right:$b_{i-2k+1}$}] (po39nt) at (2.7,4) [circle,draw,fill=black, inner sep=2pt] {};
\node[label={[font=\tiny,rotate=90]right:$a_{i-k+1}$}] (po40nt) at (3.3,4) [star,star points=5,star point ratio=2,draw,fill=black!75, inner sep=1.5pt, minimum width=1.5pt] {};
\node[label={[font=\tiny,rotate=90]right:$a_{i+k+1}$}] (po41nt) at (3.9,4) [regular polygon,regular polygon sides=5,draw,fill=white, inner sep=2pt, minimum width=2pt] {};
\node[label={[font=\tiny,rotate=90]right:$b_{i+2k+1}$}] (po42nt) at (4.5,4) [circle,draw,fill=black, inner sep=2pt] {};
\node[label={[font=\tiny,rotate=90]right:$b_{i-k+2}$}] (po43nt) at (5.1,4) [circle,draw,fill=black, inner sep=2pt] {};
\node[label={[font=\tiny,rotate=90]right:$b_{i+k+2}$}] (po44nt) at (5.7,4) [circle,draw,fill=black, inner sep=2pt] {};
\node[label={[font=\tiny,rotate=90]right:$b_{i+3}$}] (po45nt) at (6.3,4) [circle,draw,fill=black, inner sep=2pt] {};
\node[label={[font=\tiny,rotate=90]right:$a_{i+4}$}] (po46nt) at (6.9,4) [circle,draw,fill=black, inner sep=2pt] {};
\draw (po1nt)--(po2nt)--(po5nt)--(po11nt)--(po23nt);
\draw (po11nt)--(po24nt);
\draw (po5nt)--(po12nt)--(po25nt);
\draw (po12nt)--(po26nt);
\draw (po2nt)--(po6nt)--(po13nt)--(po27nt);
\draw (po13nt)--(po28nt);
\draw (po6nt)--(po14nt)--(po29nt);
\draw (po14nt)--(po30nt);
\draw (po1nt)--(po3nt)--(po7nt)--(po15nt)--(po31nt);
\draw (po15nt)--(po32nt);
\draw (po7nt)--(po16nt)--(po33nt);
\draw (po16nt)--(po34nt);
\draw (po3nt)--(po8nt)--(po17nt)--(po35nt);
\draw (po17nt)--(po36nt);
\draw (po8nt)--(po18nt)--(po37nt);
\draw (po18nt)--(po38nt);
\draw (po1nt)--(po4nt)--(po9nt)--(po19nt)--(po39nt);
\draw (po19nt)--(po40nt);
\draw (po9nt)--(po20nt)--(po41nt);
\draw (po20nt)--(po42nt);
\draw (po4nt)--(po10nt)--(po21nt)--(po43nt);
\draw (po21nt)--(po44nt);
\draw (po10nt)--(po22nt)--(po45nt);
\draw (po22nt)--(po46nt);

\end{tikzpicture}

\end{figure}

\begin{figure}\caption{The vertices within distance $4$ of $b_i$, in our families of graphs. The dotted lines represent an edge between $b_{i+3k}$ and $b_{i-3k}$.}
\label{figB}

\begin{tikzpicture}

\node[label={[label distance=-2 pt,font=\tiny]below:$b_i$}] (po1nt) at (0,-.5) [circle,draw,fill=black, inner sep=2pt] {};
\node[label={[label distance=-4 pt,font=\tiny]left:$b_{i-k}$}] (po2nt) at (-4.8,1) [circle,draw,fill=black, inner sep=2pt] {};
\node[label={[label distance=-2 pt,font=\tiny]above:$a_{i}$}] (po3nt) at (0,1) [circle,draw,fill=black, inner sep=2pt] {};
\node[label={[label distance=-3 pt,font=\tiny]right:$b_{i+k}$}] (po4nt) at (4.8,1) [circle,draw,fill=black, inner sep=2pt] {};
\node[label={[label distance=-4 pt,font=\tiny]left:$b_{i-2k}$}] (po5nt) at (-6,2) [circle,draw,fill=black, inner sep=2pt] {};
\node[label={[label distance=-4 pt,font=\tiny]left:$a_{i-k}$}] (po6nt) at (-3.6,2) [circle,draw,fill=black, inner sep=2pt] {};
\node[label={[label distance=-4 pt,font=\tiny]left:$a_{i-1}$}] (po7nt) at (-1.2,2) [circle,draw,fill=black, inner sep=2pt] {};
\node[label={[label distance=-3 pt,font=\tiny]right:$a_{i+1}$}] (po8nt) at (1.2,2) [circle,draw,fill=black, inner sep=2pt] {};
\node[label={[label distance=-3 pt,font=\tiny]right:$a_{i+k}$}] (po9nt) at (3.6,2) [circle,draw,fill=black, inner sep=2pt] {};
\node[label={[label distance=-3 pt,font=\tiny]right:$b_{i+2k}$}] (po10nt) at (6,2) [circle,draw,fill=black, inner sep=2pt] {};
\node[label={[label distance=-6 pt,font=\tiny]175:$b_{i-3k}$}] (po11nt) at (-6.6,3) [circle,draw,fill=black, inner sep=2pt] {};
\node[label={[label distance=-4 pt,font=\tiny]left:$a_{i-2k}$}] (po12nt) at (-5.4,3) [circle,draw,fill=black, inner sep=2pt] {};
\node[label={[label distance=-4 pt,font=\tiny]left:$a_{i-k-1}$}] (po13nt) at (-4.2,3) [circle,draw,fill=black, inner sep=2pt] {};
\node[label={[label distance=-4 pt,font=\tiny]left:$a_{i-k+1}$}] (po14nt) at (-3,3) [circle,draw,fill=black, inner sep=2pt] {};
\node[label={[label distance=-4 pt,font=\tiny]left:$b_{i-1}$}] (po15nt) at (-1.8,3) [circle,draw,fill=black, inner sep=2pt] {};
\node[label={[label distance=-4 pt,font=\tiny]left:$a_{i-2}$}] (po16nt) at (-.6,3) [circle,draw,fill=black, inner sep=2pt] {};
\node[label={[label distance=-3 pt,font=\tiny]right:$a_{i+2}$}] (po17nt) at (0.6,3) [circle,draw,fill=black, inner sep=2pt] {};
\node[label={[label distance=-3 pt,font=\tiny]right:$b_{i+1}$}] (po18nt) at (1.8,3) [circle,draw,fill=black, inner sep=2pt] {};
\node[label={[label distance=-3 pt,font=\tiny]right:$a_{i+k-1}$}] (po19nt) at (3,3) [circle,draw,fill=black, inner sep=2pt] {};
\node[label={[label distance=-3 pt,font=\tiny]right:$a_{i+k+1}$}] (po20nt) at (4.2,3) [circle,draw,fill=black, inner sep=2pt] {};
\node[label={[label distance=-3 pt,font=\tiny]right:$a_{i+2k}$}] (po21nt) at (5.4,3) [circle,draw,fill=black, inner sep=2pt] {};
\node[label={[label distance=-5 pt,font=\tiny]5:$b_{i+3k}$}] (po22nt) at (6.6,3) [circle,draw,fill=black, inner sep=2pt] {};
\node[label={[font=\tiny,rotate=90]right:$a_{i-3k}$}] (po24nt) at (-6.3,4) [circle,draw,fill=black, inner sep=2pt] {};
\node[label={[font=\tiny,rotate=90]right:$a_{i-2k-1}$}] (po25nt) at (-5.7,4) [circle,draw,fill=black, inner sep=2pt] {};
\node[label={[font=\tiny,rotate=90]right:$a_{i-2k+1}$}] (po26nt) at (-5.1,4) [circle,draw,fill=black, inner sep=2pt] {};
\node[label={[font=\tiny,rotate=90]right:$b_{i-k-1}$}] (po27nt) at (-4.5,4) [diamond,draw,fill=black!50, inner sep=1.75pt, minimum width=1.75pt] {};
\node[label={[font=\tiny,rotate=90]right:$a_{i-k-2}$}] (po28nt) at (-3.9,4) [circle,draw,fill=black, inner sep=2pt] {};
\node[label={[font=\tiny,rotate=90]right:$b_{i-k+1}$}] (po29nt) at (-3.3,4) [star,star points=9,star point ratio=2,draw,fill=black!25, inner sep=1.5pt, minimum width=1.5pt] {};
\node[label={[font=\tiny,rotate=90]right:$a_{i-k+2}$}] (po30nt) at (-2.7,4) [circle,draw,fill=black, inner sep=2pt] {};
\node[label={[font=\tiny,rotate=90]right:$b_{i-k-1}$}] (po31nt) at (-2.1,4) [diamond,draw,fill=black!50, inner sep=1.75pt, minimum width=1.75pt] {};
\node[label={[font=\tiny,rotate=90]right:$b_{i+k-1}$}] (po32nt) at (-1.5,4) [star,star points=5,star point ratio=2,draw,fill=black!75, inner sep=1.5pt, minimum width=1.5pt] {};
\node[label={[font=\tiny,rotate=90]right:$b_{i-2}$}] (po33nt) at (-.9,4) [circle,draw,fill=black, inner sep=2pt] {};
\node[label={[font=\tiny,rotate=90]right:$a_{i-3}$}] (po34nt) at (-.3,4) [circle,draw,fill=black, inner sep=2pt] {};
\node[label={[font=\tiny,rotate=90]right:$a_{i+3}$}] (po35nt) at (.3,4) [circle,draw,fill=black, inner sep=2pt] {};
\node[label={[font=\tiny,rotate=90]right:$b_{i+2}$}] (po36nt) at (.9,4) [circle,draw,fill=black, inner sep=2pt] {};
\node[label={[font=\tiny,rotate=90]right:$b_{i-k+1}$}] (po37nt) at (1.5,4) [star,star points=9,star point ratio=2,draw,fill=black!25, inner sep=1.5pt, minimum width=1.5pt] {};
\node[label={[font=\tiny,rotate=90]right:$b_{i+k+1}$}] (po38nt) at (2.1,4) [regular polygon,regular polygon sides=5,draw,fill=white, inner sep=2pt, minimum width=2pt] {};
\node[label={[font=\tiny,rotate=90]right:$a_{i+k-2}$}] (po39nt) at (2.7,4) [circle,draw,fill=black, inner sep=2pt] {};
\node[label={[font=\tiny,rotate=90]right:$b_{i+k-1}$}] (po40nt) at (3.3,4) [star,star points=5,star point ratio=2,draw,fill=black!75, inner sep=1.5pt, minimum width=1.5pt] {};
\node[label={[font=\tiny,rotate=90]right:$a_{i+k+2}$}] (po41nt) at (3.9,4) [circle,draw,fill=black, inner sep=2pt] {};
\node[label={[font=\tiny,rotate=90]right:$b_{i+k+1}$}] (po42nt) at (4.5,4) [regular polygon,regular polygon sides=5,draw,fill=white, inner sep=2pt, minimum width=2pt] {};
\node[label={[font=\tiny,rotate=90]right:$a_{i+2k-1}$}] (po43nt) at (5.1,4) [circle,draw,fill=black, inner sep=2pt] {};
\node[label={[font=\tiny,rotate=90]right:$a_{i+2k+1}$}] (po44nt) at (5.7,4) [circle,draw,fill=black, inner sep=2pt] {};
\node[label={[font=\tiny,rotate=90]right:$a_{i+3k}$}] (po45nt) at (6.3,4) [circle,draw,fill=black, inner sep=2pt] {};
\draw (po1nt)--(po2nt)--(po5nt)--(po11nt);
\draw (po11nt)--(po24nt);
\draw (po5nt)--(po12nt)--(po25nt);
\draw (po12nt)--(po26nt);
\draw (po2nt)--(po6nt)--(po13nt)--(po27nt);
\draw (po13nt)--(po28nt);
\draw (po6nt)--(po14nt)--(po29nt);
\draw (po14nt)--(po30nt);
\draw (po1nt)--(po3nt)--(po7nt)--(po15nt)--(po31nt);
\draw (po15nt)--(po32nt);
\draw (po7nt)--(po16nt)--(po33nt);
\draw (po16nt)--(po34nt);
\draw (po3nt)--(po8nt)--(po17nt)--(po35nt);
\draw (po17nt)--(po36nt);
\draw (po8nt)--(po18nt)--(po37nt);
\draw (po18nt)--(po38nt);
\draw (po1nt)--(po4nt)--(po9nt)--(po19nt)--(po39nt);
\draw (po19nt)--(po40nt);
\draw (po9nt)--(po20nt)--(po41nt);
\draw (po20nt)--(po42nt);
\draw (po4nt)--(po10nt)--(po21nt)--(po43nt);
\draw (po21nt)--(po44nt);
\draw (po10nt)--(po22nt)--(po45nt);
\draw [dotted, line width = 1 pt] (po22nt)--(8,3);
\draw [dotted, line width = 1 pt] (po11nt)--(-8,3);

\end{tikzpicture}

\end{figure}

\section{Main result}\label{main}

We begin by defining what it means for a robber to be trapped.

\begin{defn}
The cops have \emph{trapped} the robber if there is a cop on or adjacent to each $v_i$ ($i \in \{1,2,3\}$), where $v_1$, $v_2$, and $v_3$ are the neighbours of the vertex the robber is on (no matter whose turn it is), or if at least one cop is adjacent to the robber on the cops' turn.
\end{defn}

Observe that if the cops have trapped the robber, then the robber will be caught on the cops' next move if the robber moves, or within the next two cop moves if the robber passes, and this is the only configuration for which this is true.

We now introduce the labelling system used throughout the rest of this paper and all the proofs within. We assume throughout that the game is being played with $3$ cops, $C_1$, $C_2$, and $C_3$. We will use $r$ to denote the vertex the robber starts on. Label the three vertices adjacent to $r$ with $v_1$, $v_2$, and $v_3$. When we mention a ``branch" from $v_i$ ($i \in \{1,2,3\}$), we are referring to $v_i$ and all of the other $7$ vertices along any paths from the robber's vertex to the vertices at distance $4$ from the robber, using only a fixed one of $v_i$'s neighbours (excluding $r$ from our count). All of this is illustrated in Figure~\ref{fig:branches}. 

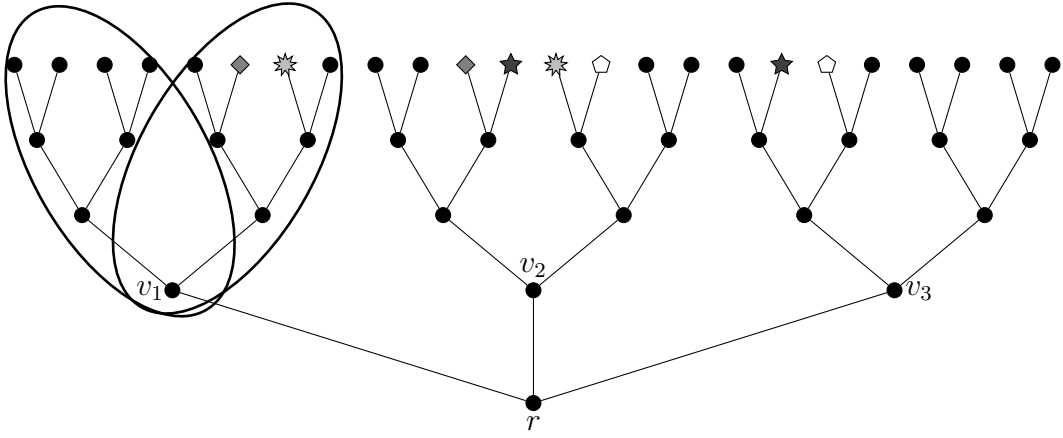
\begin{figure}\caption{This figure shows our basic notation. In the illustration, $r$ is a vertex in $A$. The two branches from $v_1$ are circled.}
\label{fig:branches}

\begin{tikzpicture}

\node[label={[label distance=-2 pt]below:$r$}] (po1nt) at (0,-.5) [circle,draw,fill=black, inner sep=2pt] {};
\node[label={[label distance=-4 pt]left:$v_1$}] (po2nt) at (-4.8,1) [circle,draw,fill=black, inner sep=2pt] {};
\node[label={[label distance=-2 pt]above:$v_2$}] (po3nt) at (0,1) [circle,draw,fill=black, inner sep=2pt] {};
\node[label={[label distance=-3 pt]right:$v_3$}] (po4nt) at (4.8,1) [circle,draw,fill=black, inner sep=2pt] {};
\node[label={[label distance=-4 pt,font=\tiny]left:}] (po5nt) at (-6,2) [circle,draw,fill=black, inner sep=2pt] {};
\node[label={[label distance=-4 pt,font=\tiny]left:}] (po6nt) at (-3.6,2) [circle,draw,fill=black, inner sep=2pt] {};
\node[label={[label distance=-4 pt,font=\tiny]left:}] (po7nt) at (-1.2,2) [circle,draw,fill=black, inner sep=2pt] {};
\node[label={[label distance=-3 pt,font=\tiny]right:}] (po8nt) at (1.2,2) [circle,draw,fill=black, inner sep=2pt] {};
\node[label={[label distance=-3 pt,font=\tiny]right:}] (po9nt) at (3.6,2) [circle,draw,fill=black, inner sep=2pt] {};
\node[label={[label distance=-3 pt,font=\tiny]right:}] (po10nt) at (6,2) [circle,draw,fill=black, inner sep=2pt] {};
\node[label={[label distance=-4 pt,font=\tiny]left:}] (po11nt) at (-6.6,3) [circle,draw,fill=black, inner sep=2pt] {};
\node[label={[label distance=-4 pt,font=\tiny]left:}] (po12nt) at (-5.4,3) [circle,draw,fill=black, inner sep=2pt] {};
\node[label={[label distance=-4 pt,font=\tiny]left:}] (po13nt) at (-4.2,3) [circle,draw,fill=black, inner sep=2pt] {};
\node[label={[label distance=-4 pt,font=\tiny]left:}] (po14nt) at (-3,3) [circle,draw,fill=black, inner sep=2pt] {};
\node[label={[label distance=-4 pt,font=\tiny]left:}] (po15nt) at (-1.8,3) [circle,draw,fill=black, inner sep=2pt] {};
\node[label={[label distance=-4 pt,font=\tiny]left:}] (po16nt) at (-.6,3) [circle,draw,fill=black, inner sep=2pt] {};
\node[label={[label distance=-3 pt,font=\tiny]right:}] (po17nt) at (0.6,3) [circle,draw,fill=black, inner sep=2pt] {};
\node[label={[label distance=-3 pt,font=\tiny]right:}] (po18nt) at (1.8,3) [circle,draw,fill=black, inner sep=2pt] {};
\node[label={[label distance=-3 pt,font=\tiny]right:}] (po19nt) at (3,3) [circle,draw,fill=black, inner sep=2pt] {};
\node[label={[label distance=-3 pt,font=\tiny]right:}] (po20nt) at (4.2,3) [circle,draw,fill=black, inner sep=2pt] {};
\node[label={[label distance=-3 pt,font=\tiny]right:}] (po21nt) at (5.4,3) [circle,draw,fill=black, inner sep=2pt] {};
\node[label={[label distance=-3 pt,font=\tiny]right:}] (po22nt) at (6.6,3) [circle,draw,fill=black, inner sep=2pt] {};
\node[label={[font=\tiny,rotate=90]right:}] (po23nt) at (-6.9,4) [circle,draw,fill=black, inner sep=2pt] {};
\node[label={[font=\tiny,rotate=90]right:}] (po24nt) at (-6.3,4) [circle,draw,fill=black, inner sep=2pt] {};
\node[label={[font=\tiny,rotate=90]right:}] (po25nt) at (-5.7,4) [circle,draw,fill=black, inner sep=2pt] {};
\node[label={[font=\tiny,rotate=90]right:}] (po26nt) at (-5.1,4) [circle,draw,fill=black, inner sep=2pt] {};
\node[label={[font=\tiny,rotate=90]right:}] (po27nt) at (-4.5,4) [circle,draw,fill=black, inner sep=2pt] {};
\node[label={[font=\tiny,rotate=90]right:}] (po28nt) at (-3.9,4) [diamond,draw,fill=black!50, inner sep=1.75pt, minimum width=1.75pt] {};
\node[label={[font=\tiny,rotate=90]right:}] (po29nt) at (-3.3,4) [star,star points=9,star point ratio=2,draw,fill=black!25, inner sep=1.5pt, minimum width=1.5pt] {};
\node[label={[font=\tiny,rotate=90]right:}] (po30nt) at (-2.7,4) [circle,draw,fill=black, inner sep=2pt] {};
\node[label={[font=\tiny,rotate=90]right:}] (po31nt) at (-2.1,4) [circle,draw,fill=black, inner sep=2pt] {};
\node[label={[font=\tiny,rotate=90]right:}] (po32nt) at (-1.5,4) [circle,draw,fill=black, inner sep=2pt] {};
\node[label={[font=\tiny,rotate=90]right:}] (po33nt) at (-.9,4) [diamond,draw,fill=black!50, inner sep=1.75pt, minimum width=1.75pt] {};
\node[label={[font=\tiny,rotate=90]right:}] (po34nt) at (-.3,4) [star,star points=5,star point ratio=2,draw,fill=black!75, inner sep=1.5pt, minimum width=1.5pt] {};
\node[label={[font=\tiny,rotate=90]right:}] (po35nt) at (.3,4) [star,star points=9,star point ratio=2,draw,fill=black!25, inner sep=1.5pt, minimum width=1.5pt] {};
\node[label={[font=\tiny,rotate=90]right:}] (po36nt) at (.9,4) [regular polygon,regular polygon sides=5,draw,fill=white, inner sep=2pt, minimum width=2pt] {};
\node[label={[font=\tiny,rotate=90]right:}] (po37nt) at (1.5,4) [circle,draw,fill=black, inner sep=2pt] {};
\node[label={[font=\tiny,rotate=90]right:}] (po38nt) at (2.1,4) [circle,draw,fill=black, inner sep=2pt] {};
\node[label={[font=\tiny,rotate=90]right:}] (po39nt) at (2.7,4) [circle,draw,fill=black, inner sep=2pt] {};
\node[label={[font=\tiny,rotate=90]right:}] (po40nt) at (3.3,4) [star,star points=5,star point ratio=2,draw,fill=black!75, inner sep=1.5pt, minimum width=1.5pt] {};
\node[label={[font=\tiny,rotate=90]right:}] (po41nt) at (3.9,4) [regular polygon,regular polygon sides=5,draw,fill=white, inner sep=2pt, minimum width=2pt] {};
\node[label={[font=\tiny,rotate=90]right:}] (po42nt) at (4.5,4) [circle,draw,fill=black, inner sep=2pt] {};
\node[label={[font=\tiny,rotate=90]right:}] (po43nt) at (5.1,4) [circle,draw,fill=black, inner sep=2pt] {};
\node[label={[font=\tiny,rotate=90]right:}] (po44nt) at (5.7,4) [circle,draw,fill=black, inner sep=2pt] {};
\node[label={[font=\tiny,rotate=90]right:}] (po45nt) at (6.3,4) [circle,draw,fill=black, inner sep=2pt] {};
\node[label={[font=\tiny,rotate=90]right:}] (po46nt) at (6.9,4) [circle,draw,fill=black, inner sep=2pt] {};
\draw (po1nt)--(po2nt)--(po5nt)--(po11nt)--(po23nt);
\draw (po11nt)--(po24nt);
\draw (po5nt)--(po12nt)--(po25nt);
\draw (po12nt)--(po26nt);
\draw (po2nt)--(po6nt)--(po13nt)--(po27nt);
\draw (po13nt)--(po28nt);
\draw (po6nt)--(po14nt)--(po29nt);
\draw (po14nt)--(po30nt);
\draw (po1nt)--(po3nt)--(po7nt)--(po15nt)--(po31nt);
\draw (po15nt)--(po32nt);
\draw (po7nt)--(po16nt)--(po33nt);
\draw (po16nt)--(po34nt);
\draw (po3nt)--(po8nt)--(po17nt)--(po35nt);
\draw (po17nt)--(po36nt);
\draw (po8nt)--(po18nt)--(po37nt);
\draw (po18nt)--(po38nt);
\draw (po1nt)--(po4nt)--(po9nt)--(po19nt)--(po39nt);
\draw (po19nt)--(po40nt);
\draw (po9nt)--(po20nt)--(po41nt);
\draw (po20nt)--(po42nt);
\draw (po4nt)--(po10nt)--(po21nt)--(po43nt);
\draw (po21nt)--(po44nt);
\draw (po10nt)--(po22nt)--(po45nt);
\draw (po22nt)--(po46nt);
\draw[rotate=-30,line width=1pt] (-4.9,.35) ellipse (33pt and 65pt);
\draw[rotate=30,line width=1pt] (-3.4,5.1) ellipse (33pt and 65pt);

\end{tikzpicture}

\end{figure}

We now outline the three cases included in \cite{MRS} as the only possible sets of configurations for the robber to be in, relative to the positions of the three cops. They are as follows:

\begin{itemize}
\item \begin{case}\label{case:clear-branch} For at least one vertex $v_i$ ($i \in \{1,2,3\}$), $v_i$ has no cop on one of its branches, and no cop within distance $2$ of the robber on the other branch. \end{case}
\item \begin{case}\label{case:cops-far} For every cop $C_i$ ($i \in \{1,2,3\}$), $C_i$ is not within distance $2$ of the robber.\end{case}
\item \begin{case}\label{case:case3} For every vertex $v_i$ ($i \in \{1,2,3\}$), $v_i$ either has at least one cop on its branches who is at distance $2$ or less from the robber (which is true for at least one vertex), or has at least one cop on each of its branches. \end{case}
\end{itemize}

Now that we have stated our three cases, we will explain why these cases cover all of the possibilities. Suppose we have a set-up that does not conform to Case~\ref{case:cops-far}. This means that there is at least one cop who is at distance $2$ or less from the robber. This cop is on $v_i$ ($i \in \{1,2,3\}$), or one of its adjacent vertices (other than $r$). If the situation also does not satisfy the assumptions of Case~\ref{case:case3}, then there must be at least one vertex (say $v_j$, where $j \in \{1,2,3\}$ and $j \neq i$) that has no cop on one of its branches, and the closest cop through $v_j$ is at distance $3$ or more from the robber. This means that Case~\ref{case:clear-branch} applies to vertex $v_j$.

We will now prove our first lemma, which will be used in our main lemma.

\begin{lem}\label{lem:not-case-3}
Suppose that $n=7k/i$ where $i \in \{1,2,3\}$, and either $n \ge 42$ or $(n,k) \in \{(28,8),(35,10),(35,15)\}$. If we play cops and robbers with $3$ cops on $GP(n,k)$, then the only way for the configuration of the robber with respect to the cops to fall into Case~\ref{case:case3} is if the robber is trapped.
\end{lem}

\begin{proof}
To prove this lemma, we will use a case analysis on the possible configurations within Case~\ref{case:case3}, and show why each is impossible unless the robber is trapped. Our proof holds regardless of whether the robber is on a vertex in $A$ or a vertex in $B$. We assume that both the cop and robber players use their moves optimally.

We now move to our case analysis.

\setcounter{case}{0}
\renewcommand{\thecase}{\Alph{case}}
\begin{case}\label{case:one-v} \textbf{\mathversion{bold} For exactly one vertex $v_i$ ($i \in \{1,2,3\}$), $v_i$ satisfies the conditions of Case~\ref{case:case3} by having at least one cop on its branches who is at distance $2$ or less from the robber.}\end{case}

In order for our situation to fall under Case~\ref{case:one-v}, $v_i$ satisfies the conditions of Case~\ref{case:case3} already, but there are two other vertices, $v_j$ and $v_k$ (where $\{i,j,k\}=\{1,2,3\}$), that still need to satisfy these conditions. This means that  each of $v_j$ and $v_k$ requires both of its branches to be covered by cops. There is no way for two cops to cover all four branches, unless those two cops are directly on $v_j$ and $v_k$, which does not fall under Case~\ref{case:one-v}. This means that it is impossible for a configuration of cops and robber on our graphs to fall under Case~\ref{case:one-v}. $\blacksquare$

\begin{case}\label{case:two-v}\textbf{\mathversion{bold}For exactly two vertices $v_i$ and $v_j$ ($i,j \in \{1,2,3\}$), $v_i$ and $v_j$ satisfy the conditions of Case~\ref{case:case3} by each having at least one cop on its branches who is at distance $2$ or less from the robber.}\end{case}

Assuming our scenario meets the parameters of Case~\ref{case:two-v}, there is only one vertex $v_k$ (where $\{i,j,k\}=\{1,2,3\}$) which does not yet satisfy the conditions of Case~\ref{case:case3}. This final vertex is required to have at least one cop on each of its branches. Since two of the three cops have already been placed, we only have one cop left. No $v_k$ has a vertex that is a part of both of its branches, other than $v_k$ itself, which cannot have a cop on it because of the conditions of Case~\ref{case:two-v}. $\blacksquare$

\begin{case}\label{case:three-v}\textbf{\mathversion{bold} For every vertex $v_i$ ($i \in \{1,2,3\}$), $v_i$ satisfies the conditions of Case~\ref{case:case3} by having at least one cop on its branches who is at distance $2$ or less from the robber.}\end{case}

In order for the parameters of Case~\ref{case:three-v} to be met, there must be a cop on each of $v_1$, $v_2$, and $v_3$, or the adjacent vertices (excluding $r$, of course). In this case, the robber is trapped. $\blacksquare$
\end{proof}

We now prove our main lemma, which will be used inductively in the proof of our theorem.

\begin{lem}\label{lem:not-2-trapped}
Suppose that $n=7k/i$ where $i \in \{1,2,3\}$, and either $n \ge 42$ or $(n,k) \in \{(28,8),(35,10),(35,15)\}$. If we play cops and robbers on $GP(n,k)$ with three cops, then unless the robber starts their turn trapped, they always have a legal move so that they cannot be trapped by the cops on the cops' turn.
\end{lem}

\begin{proof}
To prove this lemma, we will use a case analysis of the current positions of the cops compared to that of the robber (on the robber's turn). We will  use the techniques from \cite{MRS} to prove this lemma. We assume that the three cops have not yet trapped the robber. We also assume that both the cop and robber players use their moves optimally. We will go on to prove that unless the robber starts their turn trapped, they always have a legal move so that the robber cannot be trapped at the end of the cops' turn.

In Lemma \ref{lem:not-case-3}, we proved that Case~\ref{case:case3} cannot apply to our graphs. This means that we need only prove this lemma under the assumption that the scenario falls into Case~\ref{case:clear-branch} or Case~\ref{case:cops-far}. These proofs are below.

\setcounter{case}{0}
\begin{case} \textbf{\mathversion{bold}For at least one vertex $v_i$ ($i \in \{1,2,3\}$), $v_i$ has no cop on one of its branches, and no cop within distance $2$ of the robber on the other branch.}\end{case}

Let us assume that the vertex satisfying the requirements of this case is $v_1$ (the other cases are similar). Equally, let us assume that the left branch is the one with no cop. This situation is shown in Figure~\ref{fig:case1}. Since there is no cop within distance $2$ of the robber on vertex $v_1$, the robber can go to $v_1$ safely. No cop can catch the robber on this move, since no cop can be within distance $1$ of $v_1$ at the start of this turn. Since there are no cops on the left branch from $v_1$, there are no cops within distance $4$ of the robber in that direction before the robber's move, so even after the cops' turn there cannot be a cop within distance $2$ of the robber in that direction. Therefore, the robber is not trapped at the end of the cops' turn. $\blacksquare$
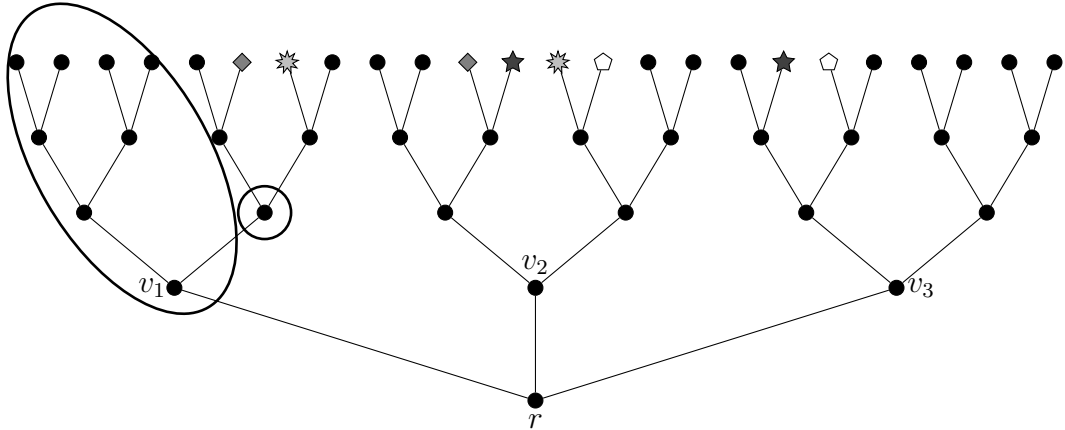
\begin{figure}\caption{The vertices that we assume do not have cops on them in Case~\ref{case:clear-branch} are circled. The illustration shows this when $r \in A$.}
\label{fig:case1}

\begin{tikzpicture}

\node[label={[label distance=-2 pt]below:$r$}] (po1nt) at (0,-.5) [circle,draw,fill=black, inner sep=2pt] {};
\node[label={[label distance=-4 pt]left:$v_1$}] (po2nt) at (-4.8,1) [circle,draw,fill=black, inner sep=2pt] {};
\node[label={[label distance=-2 pt]above:$v_2$}] (po3nt) at (0,1) [circle,draw,fill=black, inner sep=2pt] {};
\node[label={[label distance=-3 pt]right:$v_3$}] (po4nt) at (4.8,1) [circle,draw,fill=black, inner sep=2pt] {};
\node[label={[label distance=-4 pt,font=\tiny]left:}] (po5nt) at (-6,2) [circle,draw,fill=black, inner sep=2pt] {};
\node[label={[label distance=-4 pt,font=\tiny]left:}] (po6nt) at (-3.6,2) [circle,draw,fill=black, inner sep=2pt] {};
\node[label={[label distance=-4 pt,font=\tiny]left:}] (po7nt) at (-1.2,2) [circle,draw,fill=black, inner sep=2pt] {};
\node[label={[label distance=-3 pt,font=\tiny]right:}] (po8nt) at (1.2,2) [circle,draw,fill=black, inner sep=2pt] {};
\node[label={[label distance=-3 pt,font=\tiny]right:}] (po9nt) at (3.6,2) [circle,draw,fill=black, inner sep=2pt] {};
\node[label={[label distance=-3 pt,font=\tiny]right:}] (po10nt) at (6,2) [circle,draw,fill=black, inner sep=2pt] {};
\node[label={[label distance=-4 pt,font=\tiny]left:}] (po11nt) at (-6.6,3) [circle,draw,fill=black, inner sep=2pt] {};
\node[label={[label distance=-4 pt,font=\tiny]left:}] (po12nt) at (-5.4,3) [circle,draw,fill=black, inner sep=2pt] {};
\node[label={[label distance=-4 pt,font=\tiny]left:}] (po13nt) at (-4.2,3) [circle,draw,fill=black, inner sep=2pt] {};
\node[label={[label distance=-4 pt,font=\tiny]left:}] (po14nt) at (-3,3) [circle,draw,fill=black, inner sep=2pt] {};
\node[label={[label distance=-4 pt,font=\tiny]left:}] (po15nt) at (-1.8,3) [circle,draw,fill=black, inner sep=2pt] {};
\node[label={[label distance=-4 pt,font=\tiny]left:}] (po16nt) at (-.6,3) [circle,draw,fill=black, inner sep=2pt] {};
\node[label={[label distance=-3 pt,font=\tiny]right:}] (po17nt) at (0.6,3) [circle,draw,fill=black, inner sep=2pt] {};
\node[label={[label distance=-3 pt,font=\tiny]right:}] (po18nt) at (1.8,3) [circle,draw,fill=black, inner sep=2pt] {};
\node[label={[label distance=-3 pt,font=\tiny]right:}] (po19nt) at (3,3) [circle,draw,fill=black, inner sep=2pt] {};
\node[label={[label distance=-3 pt,font=\tiny]right:}] (po20nt) at (4.2,3) [circle,draw,fill=black, inner sep=2pt] {};
\node[label={[label distance=-3 pt,font=\tiny]right:}] (po21nt) at (5.4,3) [circle,draw,fill=black, inner sep=2pt] {};
\node[label={[label distance=-3 pt,font=\tiny]right:}] (po22nt) at (6.6,3) [circle,draw,fill=black, inner sep=2pt] {};
\node[label={[font=\tiny,rotate=90]right:}] (po23nt) at (-6.9,4) [circle,draw,fill=black, inner sep=2pt] {};
\node[label={[font=\tiny,rotate=90]right:}] (po24nt) at (-6.3,4) [circle,draw,fill=black, inner sep=2pt] {};
\node[label={[font=\tiny,rotate=90]right:}] (po25nt) at (-5.7,4) [circle,draw,fill=black, inner sep=2pt] {};
\node[label={[font=\tiny,rotate=90]right:}] (po26nt) at (-5.1,4) [circle,draw,fill=black, inner sep=2pt] {};
\node[label={[font=\tiny,rotate=90]right:}] (po27nt) at (-4.5,4) [circle,draw,fill=black, inner sep=2pt] {};
\node[label={[font=\tiny,rotate=90]right:}] (po28nt) at (-3.9,4) [diamond,draw,fill=black!50, inner sep=1.75pt, minimum width=1.75pt] {};
\node[label={[font=\tiny,rotate=90]right:}] (po29nt) at (-3.3,4) [star,star points=9,star point ratio=2,draw,fill=black!25, inner sep=1.5pt, minimum width=1.5pt] {};
\node[label={[font=\tiny,rotate=90]right:}] (po30nt) at (-2.7,4) [circle,draw,fill=black, inner sep=2pt] {};
\node[label={[font=\tiny,rotate=90]right:}] (po31nt) at (-2.1,4) [circle,draw,fill=black, inner sep=2pt] {};
\node[label={[font=\tiny,rotate=90]right:}] (po32nt) at (-1.5,4) [circle,draw,fill=black, inner sep=2pt] {};
\node[label={[font=\tiny,rotate=90]right:}] (po33nt) at (-.9,4) [diamond,draw,fill=black!50, inner sep=1.75pt, minimum width=1.75pt] {};
\node[label={[font=\tiny,rotate=90]right:}] (po34nt) at (-.3,4) [star,star points=5,star point ratio=2,draw,fill=black!75, inner sep=1.5pt, minimum width=1.5pt] {};
\node[label={[font=\tiny,rotate=90]right:}] (po35nt) at (.3,4) [star,star points=9,star point ratio=2,draw,fill=black!25, inner sep=1.5pt, minimum width=1.5pt] {};
\node[label={[font=\tiny,rotate=90]right:}] (po36nt) at (.9,4) [regular polygon,regular polygon sides=5,draw,fill=white, inner sep=2pt, minimum width=2pt] {};
\node[label={[font=\tiny,rotate=90]right:}] (po37nt) at (1.5,4) [circle,draw,fill=black, inner sep=2pt] {};
\node[label={[font=\tiny,rotate=90]right:}] (po38nt) at (2.1,4) [circle,draw,fill=black, inner sep=2pt] {};
\node[label={[font=\tiny,rotate=90]right:}] (po39nt) at (2.7,4) [circle,draw,fill=black, inner sep=2pt] {};
\node[label={[font=\tiny,rotate=90]right:}] (po40nt) at (3.3,4) [star,star points=5,star point ratio=2,draw,fill=black!75, inner sep=1.5pt, minimum width=1.5pt] {};
\node[label={[font=\tiny,rotate=90]right:}] (po41nt) at (3.9,4) [regular polygon,regular polygon sides=5,draw,fill=white, inner sep=2pt, minimum width=2pt] {};
\node[label={[font=\tiny,rotate=90]right:}] (po42nt) at (4.5,4) [circle,draw,fill=black, inner sep=2pt] {};
\node[label={[font=\tiny,rotate=90]right:}] (po43nt) at (5.1,4) [circle,draw,fill=black, inner sep=2pt] {};
\node[label={[font=\tiny,rotate=90]right:}] (po44nt) at (5.7,4) [circle,draw,fill=black, inner sep=2pt] {};
\node[label={[font=\tiny,rotate=90]right:}] (po45nt) at (6.3,4) [circle,draw,fill=black, inner sep=2pt] {};
\node[label={[font=\tiny,rotate=90]right:}] (po46nt) at (6.9,4) [circle,draw,fill=black, inner sep=2pt] {};
\draw (po1nt)--(po2nt)--(po5nt)--(po11nt)--(po23nt);
\draw (po11nt)--(po24nt);
\draw (po5nt)--(po12nt)--(po25nt);
\draw (po12nt)--(po26nt);
\draw (po2nt)--(po6nt)--(po13nt)--(po27nt);
\draw (po13nt)--(po28nt);
\draw (po6nt)--(po14nt)--(po29nt);
\draw (po14nt)--(po30nt);
\draw (po1nt)--(po3nt)--(po7nt)--(po15nt)--(po31nt);
\draw (po15nt)--(po32nt);
\draw (po7nt)--(po16nt)--(po33nt);
\draw (po16nt)--(po34nt);
\draw (po3nt)--(po8nt)--(po17nt)--(po35nt);
\draw (po17nt)--(po36nt);
\draw (po8nt)--(po18nt)--(po37nt);
\draw (po18nt)--(po38nt);
\draw (po1nt)--(po4nt)--(po9nt)--(po19nt)--(po39nt);
\draw (po19nt)--(po40nt);
\draw (po9nt)--(po20nt)--(po41nt);
\draw (po20nt)--(po42nt);
\draw (po4nt)--(po10nt)--(po21nt)--(po43nt);
\draw (po21nt)--(po44nt);
\draw (po10nt)--(po22nt)--(po45nt);
\draw (po22nt)--(po46nt);
\draw[rotate=30,line width=1pt] (-3.4,5.1) ellipse (33pt and 65pt);
\draw[line width=1pt] (-3.6,2) circle (10pt);
\end{tikzpicture}

\end{figure}

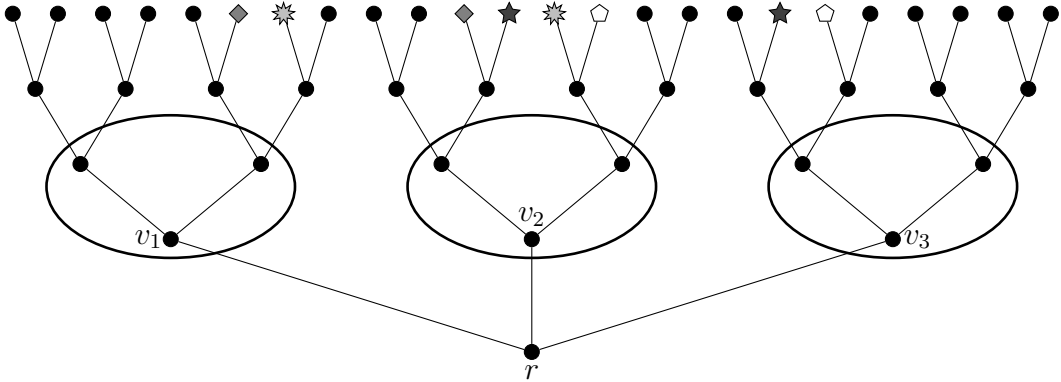
\begin{figure}\caption{The vertices that cannot have cops on them in Case~\ref{case:cops-far} are circled. The illustration shows this when  $r\in A$.}
\label{fig:case2}

\begin{tikzpicture}

\node[label={[label distance=-2 pt]below:$r$}] (po1nt) at (0,-.5) [circle,draw,fill=black, inner sep=2pt] {};
\node[label={[label distance=-4 pt]left:$v_1$}] (po2nt) at (-4.8,1) [circle,draw,fill=black, inner sep=2pt] {};
\node[label={[label distance=-2 pt]above:$v_2$}] (po3nt) at (0,1) [circle,draw,fill=black, inner sep=2pt] {};
\node[label={[label distance=-3 pt]right:$v_3$}] (po4nt) at (4.8,1) [circle,draw,fill=black, inner sep=2pt] {};
\node[label={[label distance=-4 pt,font=\tiny]left:}] (po5nt) at (-6,2) [circle,draw,fill=black, inner sep=2pt] {};
\node[label={[label distance=-4 pt,font=\tiny]left:}] (po6nt) at (-3.6,2) [circle,draw,fill=black, inner sep=2pt] {};
\node[label={[label distance=-4 pt,font=\tiny]left:}] (po7nt) at (-1.2,2) [circle,draw,fill=black, inner sep=2pt] {};
\node[label={[label distance=-3 pt,font=\tiny]right:}] (po8nt) at (1.2,2) [circle,draw,fill=black, inner sep=2pt] {};
\node[label={[label distance=-3 pt,font=\tiny]right:}] (po9nt) at (3.6,2) [circle,draw,fill=black, inner sep=2pt] {};
\node[label={[label distance=-3 pt,font=\tiny]right:}] (po10nt) at (6,2) [circle,draw,fill=black, inner sep=2pt] {};
\node[label={[label distance=-4 pt,font=\tiny]left:}] (po11nt) at (-6.6,3) [circle,draw,fill=black, inner sep=2pt] {};
\node[label={[label distance=-4 pt,font=\tiny]left:}] (po12nt) at (-5.4,3) [circle,draw,fill=black, inner sep=2pt] {};
\node[label={[label distance=-4 pt,font=\tiny]left:}] (po13nt) at (-4.2,3) [circle,draw,fill=black, inner sep=2pt] {};
\node[label={[label distance=-4 pt,font=\tiny]left:}] (po14nt) at (-3,3) [circle,draw,fill=black, inner sep=2pt] {};
\node[label={[label distance=-4 pt,font=\tiny]left:}] (po15nt) at (-1.8,3) [circle,draw,fill=black, inner sep=2pt] {};
\node[label={[label distance=-4 pt,font=\tiny]left:}] (po16nt) at (-.6,3) [circle,draw,fill=black, inner sep=2pt] {};
\node[label={[label distance=-3 pt,font=\tiny]right:}] (po17nt) at (0.6,3) [circle,draw,fill=black, inner sep=2pt] {};
\node[label={[label distance=-3 pt,font=\tiny]right:}] (po18nt) at (1.8,3) [circle,draw,fill=black, inner sep=2pt] {};
\node[label={[label distance=-3 pt,font=\tiny]right:}] (po19nt) at (3,3) [circle,draw,fill=black, inner sep=2pt] {};
\node[label={[label distance=-3 pt,font=\tiny]right:}] (po20nt) at (4.2,3) [circle,draw,fill=black, inner sep=2pt] {};
\node[label={[label distance=-3 pt,font=\tiny]right:}] (po21nt) at (5.4,3) [circle,draw,fill=black, inner sep=2pt] {};
\node[label={[label distance=-3 pt,font=\tiny]right:}] (po22nt) at (6.6,3) [circle,draw,fill=black, inner sep=2pt] {};
\node[label={[font=\tiny,rotate=90]right:}] (po23nt) at (-6.9,4) [circle,draw,fill=black, inner sep=2pt] {};
\node[label={[font=\tiny,rotate=90]right:}] (po24nt) at (-6.3,4) [circle,draw,fill=black, inner sep=2pt] {};
\node[label={[font=\tiny,rotate=90]right:}] (po25nt) at (-5.7,4) [circle,draw,fill=black, inner sep=2pt] {};
\node[label={[font=\tiny,rotate=90]right:}] (po26nt) at (-5.1,4) [circle,draw,fill=black, inner sep=2pt] {};
\node[label={[font=\tiny,rotate=90]right:}] (po27nt) at (-4.5,4) [circle,draw,fill=black, inner sep=2pt] {};
\node[label={[font=\tiny,rotate=90]right:}] (po28nt) at (-3.9,4) [diamond,draw,fill=black!50, inner sep=1.75pt, minimum width=1.75pt] {};
\node[label={[font=\tiny,rotate=90]right:}] (po29nt) at (-3.3,4) [star,star points=9,star point ratio=2,draw,fill=black!25, inner sep=1.5pt, minimum width=1.5pt] {};
\node[label={[font=\tiny,rotate=90]right:}] (po30nt) at (-2.7,4) [circle,draw,fill=black, inner sep=2pt] {};
\node[label={[font=\tiny,rotate=90]right:}] (po31nt) at (-2.1,4) [circle,draw,fill=black, inner sep=2pt] {};
\node[label={[font=\tiny,rotate=90]right:}] (po32nt) at (-1.5,4) [circle,draw,fill=black, inner sep=2pt] {};
\node[label={[font=\tiny,rotate=90]right:}] (po33nt) at (-.9,4) [diamond,draw,fill=black!50, inner sep=1.75pt, minimum width=1.75pt] {};
\node[label={[font=\tiny,rotate=90]right:}] (po34nt) at (-.3,4) [star,star points=5,star point ratio=2,draw,fill=black!75, inner sep=1.5pt, minimum width=1.5pt] {};
\node[label={[font=\tiny,rotate=90]right:}] (po35nt) at (.3,4) [star,star points=9,star point ratio=2,draw,fill=black!25, inner sep=1.5pt, minimum width=1.5pt] {};
\node[label={[font=\tiny,rotate=90]right:}] (po36nt) at (.9,4) [regular polygon,regular polygon sides=5,draw,fill=white, inner sep=2pt, minimum width=2pt] {};
\node[label={[font=\tiny,rotate=90]right:}] (po37nt) at (1.5,4) [circle,draw,fill=black, inner sep=2pt] {};
\node[label={[font=\tiny,rotate=90]right:}] (po38nt) at (2.1,4) [circle,draw,fill=black, inner sep=2pt] {};
\node[label={[font=\tiny,rotate=90]right:}] (po39nt) at (2.7,4) [circle,draw,fill=black, inner sep=2pt] {};
\node[label={[font=\tiny,rotate=90]right:}] (po40nt) at (3.3,4) [star,star points=5,star point ratio=2,draw,fill=black!75, inner sep=1.5pt, minimum width=1.5pt] {};
\node[label={[font=\tiny,rotate=90]right:}] (po41nt) at (3.9,4) [regular polygon,regular polygon sides=5,draw,fill=white, inner sep=2pt, minimum width=2pt] {};
\node[label={[font=\tiny,rotate=90]right:}] (po42nt) at (4.5,4) [circle,draw,fill=black, inner sep=2pt] {};
\node[label={[font=\tiny,rotate=90]right:}] (po43nt) at (5.1,4) [circle,draw,fill=black, inner sep=2pt] {};
\node[label={[font=\tiny,rotate=90]right:}] (po44nt) at (5.7,4) [circle,draw,fill=black, inner sep=2pt] {};
\node[label={[font=\tiny,rotate=90]right:}] (po45nt) at (6.3,4) [circle,draw,fill=black, inner sep=2pt] {};
\node[label={[font=\tiny,rotate=90]right:}] (po46nt) at (6.9,4) [circle,draw,fill=black, inner sep=2pt] {};
\draw (po1nt)--(po2nt)--(po5nt)--(po11nt)--(po23nt);
\draw (po11nt)--(po24nt);
\draw (po5nt)--(po12nt)--(po25nt);
\draw (po12nt)--(po26nt);
\draw (po2nt)--(po6nt)--(po13nt)--(po27nt);
\draw (po13nt)--(po28nt);
\draw (po6nt)--(po14nt)--(po29nt);
\draw (po14nt)--(po30nt);
\draw (po1nt)--(po3nt)--(po7nt)--(po15nt)--(po31nt);
\draw (po15nt)--(po32nt);
\draw (po7nt)--(po16nt)--(po33nt);
\draw (po16nt)--(po34nt);
\draw (po3nt)--(po8nt)--(po17nt)--(po35nt);
\draw (po17nt)--(po36nt);
\draw (po8nt)--(po18nt)--(po37nt);
\draw (po18nt)--(po38nt);
\draw (po1nt)--(po4nt)--(po9nt)--(po19nt)--(po39nt);
\draw (po19nt)--(po40nt);
\draw (po9nt)--(po20nt)--(po41nt);
\draw (po20nt)--(po42nt);
\draw (po4nt)--(po10nt)--(po21nt)--(po43nt);
\draw (po21nt)--(po44nt);
\draw (po10nt)--(po22nt)--(po45nt);
\draw (po22nt)--(po46nt);
\draw[line width=1pt] (-4.8,1.7) ellipse (47pt and 27pt);
\draw[line width=1pt] (0,1.7) ellipse (47pt and 27pt);
\draw[line width=1pt] (4.8,1.7) ellipse (47pt and 27pt);

\end{tikzpicture}

\end{figure}

\begin{case}\textbf{For every cop \mathversion{bold}$C_i$ ($i \in \{1,2,3\}$), $C_i$ is not within distance $2$ of the robber.}\end{case}

Figure~\ref{fig:case2} illustrates this case. In this scenario, the robber can move to any of the three vertices adjacent to their position ($v_1$, $v_2$, or $v_3$). Let's have the robber choose $v_1$. Since no cops were within distance $2$ of the robber at the beginning, once the robber has moved, they cannot be within distance $1$ of the robber, so cannot catch the robber. Furthermore, since no cop was on vertex $v_1$, $v_2$, $v_3$, or any of their neighbours, after the cops' move no cop is on $r$ or any of its neighbours (and $r$ is a neighbour of the robber's new position). This means that the robber is not trapped. $\blacksquare$

We have now proven that the robber cannot be trapped at the end of the cops' turn if they did not start their turn trapped.
\end{proof}

This allows us to prove our main result.

\begin{thm}
Suppose that $n=7k/i$ where $i \in \{1,2,3\}$, and $n \ge 42$ or $(n,k) \in \{(28,8),(35,10),(35,15)\}$. Then the cop number of the graph $GP(n,k)$ is $4$.
\end{thm}

\begin{proof}
We begin by showing that there is always a vertex the robber can choose for their first move, so that they do not begin the game trapped.

Let $w$ be an arbitrary vertex with neighbours $v_1$, $v_2$, and $v_3$. Recall that for the robber to be trapped, there must be cops on each of its neighbours or their adjacent vertices, or there must be at least one cop directly adjacent to the robber. In order to maximise the cops' chances of having trapped the robber, a cop should be placed on each of $v_1$, $v_2$, and $v_3$ or their adjacent vertices. Without loss of generality, let us assume that $C_i$ ($i \in \{1,2,3\}$) is on $v_i$ or an adjacent vertex. 

Suppose momentarily that all three cops choose to go on vertices adjacent to $w$. Under this scenario, should the robber go on any vertex adjacent to a cop, they will be trapped. This gives us $10$ vertices where the robber cannot go without being trapped (including those with cops already on them). 

Now, suppose that all three cops choose to go on vertices at distance $2$ from $w$. In this set-up, in addition to the robber being trapped if they are adjacent to a cop, they will also be trapped if they go on $w$, since there are cops adjacent to $v_1$, $v_2$, and $v_3$. Therefore, in this case, there are $13$ different vertices where the robber cannot go without being trapped. 

\begin{figure}\caption{The circles in this figure show the ideal positions for the cops at the beginning of the game, when $w \in A$.}
\label{figThm}

\begin{tikzpicture}

\node[label={[label distance=-2 pt]below:$w$}] (po1nt) at (0,-.5) [circle,draw,fill=black, inner sep=2pt] {};
\node[label={[label distance=-4 pt]left:$v_1$}] (po2nt) at (-4.8,1) [circle,draw,fill=black, inner sep=2pt] {};
\node[label={[label distance=-2 pt]above:$v_2$}] (po3nt) at (0,1) [circle,draw,fill=black, inner sep=2pt] {};
\node[label={[label distance=-3 pt]right:$v_3$}] (po4nt) at (4.8,1) [circle,draw,fill=black, inner sep=2pt] {};
\node[label={[label distance=-4 pt,font=\tiny]left:}] (po5nt) at (-6,2) [circle,draw,fill=black, inner sep=2pt] {};
\node[label={[label distance=-4 pt,font=\tiny]left:}] (po6nt) at (-3.6,2) [circle,draw,fill=black, inner sep=2pt] {};
\node[label={[label distance=-4 pt,font=\tiny]left:}] (po7nt) at (-1.2,2) [circle,draw,fill=black, inner sep=2pt] {};
\node[label={[label distance=-3 pt,font=\tiny]right:}] (po8nt) at (1.2,2) [circle,draw,fill=black, inner sep=2pt] {};
\node[label={[label distance=-3 pt,font=\tiny]right:}] (po9nt) at (3.6,2) [circle,draw,fill=black, inner sep=2pt] {};
\node[label={[label distance=-3 pt,font=\tiny]right:}] (po10nt) at (6,2) [circle,draw,fill=black, inner sep=2pt] {};
\node[label={[label distance=-4 pt,font=\tiny]left:}] (po11nt) at (-6.6,3) [circle,draw,fill=black, inner sep=2pt] {};
\node[label={[label distance=-4 pt,font=\tiny]left:}] (po12nt) at (-5.4,3) [circle,draw,fill=black, inner sep=2pt] {};
\node[label={[label distance=-4 pt,font=\tiny]left:}] (po13nt) at (-4.2,3) [circle,draw,fill=black, inner sep=2pt] {};
\node[label={[label distance=-4 pt,font=\tiny]left:}] (po14nt) at (-3,3) [circle,draw,fill=black, inner sep=2pt] {};
\node[label={[label distance=-4 pt,font=\tiny]left:}] (po15nt) at (-1.8,3) [circle,draw,fill=black, inner sep=2pt] {};
\node[label={[label distance=-4 pt,font=\tiny]left:}] (po16nt) at (-.6,3) [circle,draw,fill=black, inner sep=2pt] {};
\node[label={[label distance=-3 pt,font=\tiny]right:}] (po17nt) at (0.6,3) [circle,draw,fill=black, inner sep=2pt] {};
\node[label={[label distance=-3 pt,font=\tiny]right:}] (po18nt) at (1.8,3) [circle,draw,fill=black, inner sep=2pt] {};
\node[label={[label distance=-3 pt,font=\tiny]right:}] (po19nt) at (3,3) [circle,draw,fill=black, inner sep=2pt] {};
\node[label={[label distance=-3 pt,font=\tiny]right:}] (po20nt) at (4.2,3) [circle,draw,fill=black, inner sep=2pt] {};
\node[label={[label distance=-3 pt,font=\tiny]right:}] (po21nt) at (5.4,3) [circle,draw,fill=black, inner sep=2pt] {};
\node[label={[label distance=-3 pt,font=\tiny]right:}] (po22nt) at (6.6,3) [circle,draw,fill=black, inner sep=2pt] {};
\node[label={[font=\tiny,rotate=90]right:}] (po23nt) at (-6.9,4) [circle,draw,fill=black, inner sep=2pt] {};
\node[label={[font=\tiny,rotate=90]right:}] (po24nt) at (-6.3,4) [circle,draw,fill=black, inner sep=2pt] {};
\node[label={[font=\tiny,rotate=90]right:}] (po25nt) at (-5.7,4) [circle,draw,fill=black, inner sep=2pt] {};
\node[label={[font=\tiny,rotate=90]right:}] (po26nt) at (-5.1,4) [circle,draw,fill=black, inner sep=2pt] {};
\node[label={[font=\tiny,rotate=90]right:}] (po27nt) at (-4.5,4) [circle,draw,fill=black, inner sep=2pt] {};
\node[label={[font=\tiny,rotate=90]right:}] (po28nt) at (-3.9,4) [diamond,draw,fill=black!50, inner sep=1.75pt, minimum width=1.75pt] {};
\node[label={[font=\tiny,rotate=90]right:}] (po29nt) at (-3.3,4) [star,star points=9,star point ratio=2,draw,fill=black!25, inner sep=1.5pt, minimum width=1.5pt] {};
\node[label={[font=\tiny,rotate=90]right:}] (po30nt) at (-2.7,4) [circle,draw,fill=black, inner sep=2pt] {};
\node[label={[font=\tiny,rotate=90]right:}] (po31nt) at (-2.1,4) [circle,draw,fill=black, inner sep=2pt] {};
\node[label={[font=\tiny,rotate=90]right:}] (po32nt) at (-1.5,4) [circle,draw,fill=black, inner sep=2pt] {};
\node[label={[font=\tiny,rotate=90]right:}] (po33nt) at (-.9,4) [diamond,draw,fill=black!50, inner sep=1.75pt, minimum width=1.75pt] {};
\node[label={[font=\tiny,rotate=90]right:}] (po34nt) at (-.3,4) [star,star points=5,star point ratio=2,draw,fill=black!75, inner sep=1.5pt, minimum width=1.5pt] {};
\node[label={[font=\tiny,rotate=90]right:}] (po35nt) at (.3,4) [star,star points=9,star point ratio=2,draw,fill=black!25, inner sep=1.5pt, minimum width=1.5pt] {};
\node[label={[font=\tiny,rotate=90]right:}] (po36nt) at (.9,4) [regular polygon,regular polygon sides=5,draw,fill=white, inner sep=2pt, minimum width=2pt] {};
\node[label={[font=\tiny,rotate=90]right:}] (po37nt) at (1.5,4) [circle,draw,fill=black, inner sep=2pt] {};
\node[label={[font=\tiny,rotate=90]right:}] (po38nt) at (2.1,4) [circle,draw,fill=black, inner sep=2pt] {};
\node[label={[font=\tiny,rotate=90]right:}] (po39nt) at (2.7,4) [circle,draw,fill=black, inner sep=2pt] {};
\node[label={[font=\tiny,rotate=90]right:}] (po40nt) at (3.3,4) [star,star points=5,star point ratio=2,draw,fill=black!75, inner sep=1.5pt, minimum width=1.5pt] {};
\node[label={[font=\tiny,rotate=90]right:}] (po41nt) at (3.9,4) [regular polygon,regular polygon sides=5,draw,fill=white, inner sep=2pt, minimum width=2pt] {};
\node[label={[font=\tiny,rotate=90]right:}] (po42nt) at (4.5,4) [circle,draw,fill=black, inner sep=2pt] {};
\node[label={[font=\tiny,rotate=90]right:}] (po43nt) at (5.1,4) [circle,draw,fill=black, inner sep=2pt] {};
\node[label={[font=\tiny,rotate=90]right:}] (po44nt) at (5.7,4) [circle,draw,fill=black, inner sep=2pt] {};
\node[label={[font=\tiny,rotate=90]right:}] (po45nt) at (6.3,4) [circle,draw,fill=black, inner sep=2pt] {};
\node[label={[font=\tiny,rotate=90]right:}] (po46nt) at (6.9,4) [circle,draw,fill=black, inner sep=2pt] {};
\node (labe4) at (-4.8,2){$C_1$};
\node (labe5) at (0,2){$C_2$};
\node (labe6) at (4.8,2){$C_3$};
\draw (po1nt)--(po2nt)--(po5nt)--(po11nt)--(po23nt);
\draw (po11nt)--(po24nt);
\draw (po5nt)--(po12nt)--(po25nt);
\draw (po12nt)--(po26nt);
\draw (po2nt)--(po6nt)--(po13nt)--(po27nt);
\draw (po13nt)--(po28nt);
\draw (po6nt)--(po14nt)--(po29nt);
\draw (po14nt)--(po30nt);
\draw (po1nt)--(po3nt)--(po7nt)--(po15nt)--(po31nt);
\draw (po15nt)--(po32nt);
\draw (po7nt)--(po16nt)--(po33nt);
\draw (po16nt)--(po34nt);
\draw (po3nt)--(po8nt)--(po17nt)--(po35nt);
\draw (po17nt)--(po36nt);
\draw (po8nt)--(po18nt)--(po37nt);
\draw (po18nt)--(po38nt);
\draw (po1nt)--(po4nt)--(po9nt)--(po19nt)--(po39nt);
\draw (po19nt)--(po40nt);
\draw (po9nt)--(po20nt)--(po41nt);
\draw (po20nt)--(po42nt);
\draw (po4nt)--(po10nt)--(po21nt)--(po43nt);
\draw (po21nt)--(po44nt);
\draw (po10nt)--(po22nt)--(po45nt);
\draw (po22nt)--(po46nt);
\draw[line width=1pt] (-4.8,2) ellipse (47pt and 15pt);
\draw[line width=1pt] (0,2) ellipse (47pt and 15pt);
\draw[line width=1pt] (4.8,2) ellipse (47pt and 15pt);

\end{tikzpicture}

\end{figure}
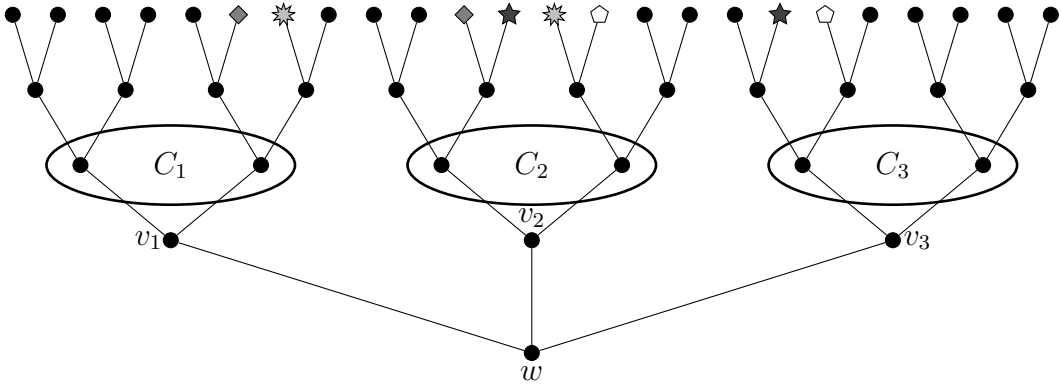

This means that our second set of positions, as shown in Figure~\ref{figThm}, are the ideal choices for the cops when there are only three of them, and if they go in the ideal positions, they can still only ensure that the robber will be trapped if placed on one of $13$ different vertices. Since $n \ge 28$, our graph has at least $56$ vertices, and the robber still has many choices that do not leave them trapped, meaning that the robber will not start trapped.

Combining this with the results of Lemma~\ref{lem:not-2-trapped}, we conclude that the robber never has to become trapped, so $c(G) > 3$.

By \cite{ball}, if $G$ is a generalised Petersen graph, then $c(G) \le 4$, so $c(G) = 4$.
\end{proof}

\textsc{Acknowledgements.}  The authors thank Dave Morris for his careful reading of their paper and his helpful suggestions.

\end{document}